\documentclass[10pt,leqno]{amsart}
\usepackage{graphicx}
\usepackage{indentfirst}

\usepackage[T1]{fontenc}
\usepackage[utf8]{inputenc}
\usepackage{lmodern}
\usepackage{amssymb,amsthm,amsmath,mathtools}
\usepackage{xcolor,paralist,hyperref,titlesec,fancyhdr,etoolbox}
\usepackage{mathrsfs}
\usepackage{enumitem}
\usepackage{microtype}
\allowdisplaybreaks

\newtheorem{theorem}{Theorem}[section]

\newtheorem{lemma}[theorem]{Lemma}
\newtheorem{proposition}[theorem]{Proposition}
\newtheorem{corollary}[theorem]{Corollary}

\theoremstyle{remark}
\newtheorem{remark}[theorem]{Remark}

\titleformat{\section}[block]{\normalfont\Large\bfseries\centering}{\thesection.}{.5em}{}
\titlespacing*{\section}{0pt}{2.0ex plus .3ex minus .2ex}{1.2ex plus .2ex}

\hypersetup{ colorlinks=true, linkcolor=black, citecolor=black, filecolor=black, urlcolor=black }
\numberwithin{equation}{section}

\newcommand{\AuthorOneName}{Yaoran Yang}
\newcommand{\AuthorOneShort}{Y. Yang}
\newcommand{\AuthorOneAffiliation}{School of Mathematics, Sichuan University, Chengdu, 610065, China}
\newcommand{\AuthorOneEmail}{yangyaoran@stu.scu.edu.cn}
\newcommand{\AuthorOneORCID}{ORCID iD: 0009-0004-2832-9163}

\newcommand{\AuthorTwoName}{Yutong Zhang}
\newcommand{\AuthorTwoShort}{Y. Zhang}
\newcommand{\AuthorTwoAffiliation}{School of Mathematics, Sichuan University, Chengdu, 610065, China}
\newcommand{\AuthorTwoEmail}{yutongzhang@stu.scu.edu.cn}
\newcommand{\AuthorTwoORCID}{ORCID iD: 0009-0000-1220-0702; corresponding author}

\newcommand{\Fp}{\mathbb F_p}
\newcommand{\Z}{\mathbb Z}
\newcommand{\Q}{\mathbb Q}
\newcommand{\C}{\mathbb C}
\newcommand{\eps}{\varepsilon}
\newcommand{\leg}[2]{\left(\frac{#1}{#2}\right)}

\newcommand{\adj}{\operatorname{adj}}
\newcommand{\diag}{\operatorname{diag}}

\newcommand{\tr}{\operatorname{tr}}
\newcommand{\QR}{\operatorname{QR}}

\newcommand{\bone}{\mathbf 1}

\newcommand{\calU}{\mathcal U}

\begin{document}

\title[Two determinant evaluations in Sun's conjectures]{Two determinant evaluations in Sun's conjectures involving Legendre symbols}

\author[\AuthorOneShort]{\AuthorOneName}
\address{\AuthorOneAffiliation}
\email{\AuthorOneEmail}
\thanks{\AuthorOneORCID}

\author[\AuthorTwoShort]{\AuthorTwoName}
\address{\AuthorTwoAffiliation}
\email{\AuthorTwoEmail}
\thanks{\AuthorTwoORCID}

\date{\today}

\begin{abstract}
We prove two determinant evaluations attached to Sun's conjectures on
matrices of Legendre symbols.  The first one resolves the
\(p\equiv1\pmod4\) part of Conjecture 4.8(i) by reducing the determinant
with four indeterminates to a four-entry inverse package for the adjacent
minor
\([\chi(j-k+1)]_{0\le j,k<(p-1)/2}\).  The core evaluation is
\[
\det H=\leg{2}{p}(b'_p-a'_p),\qquad
U^TH^{-1}U=
\begin{pmatrix}
\leg{2}{p}\dfrac{pb'_p-a'_p}{b'_p-a'_p}&1\\[2mm]
\dfrac{b'_p-a'_p-1}{b'_p-a'_p}&1
\end{pmatrix},
\]
where \(U=(\mathbf1,\eta)\) and \(\eta_j=\chi(j)\).  The proof uses
Vsemirnov's factorisation of Chapman's matrix and an adjacent cofactor
calculation.  The second result gives a uniform exact congruence modulo
\(p\) for the determinant underlying Sun's Conjecture 4.10(i), valid for
any ordered half-system modulo sign and all \(u,v\in\mathbb F_p\).  Its
standard specialization recovers the asserted square class.  The
square-class assertion itself also follows from Sun's earlier evaluation of
\(T(d,p)\); the contribution here is an exact and half-system refinement.
\end{abstract}

\maketitle

\section{Introduction and main results}

Let \(p\) be an odd prime and let
\(
\chi(t)=\leg{t}{p}
\)
be the Legendre symbol, with \(\chi(0)=0\).  Determinants with character
entries go back at least to Lehmer's character matrices and Carlitz's
cyclotomic matrices \cite{Lehmer1956,Carlitz1959}.  The modern form used
here is closer to the Chapman--Vsemirnov line of Legendre-symbol matrices
\cite{Chapman2004,Vsemirnov2012,Vsemirnov2013}.  Sun's work on determinants
with Legendre-symbol entries, especially \cite{Sun2019,Sun2024}, contains a
large collection of exact determinant conjectures and square-class
conjectures.  This paper gives a unified, appendix-free presentation of two
\(p\equiv1\pmod4\) evaluations: an exact four-variable identity from
Conjecture 4.8 and an exact uniform refinement of the determinant appearing
in Conjecture 4.10(i).

The first result belongs to the Chapman--Vsemirnov determinant method.  For
standard background on Gauss sums, Jacobi sums, and the class-number
normalisations used below we refer to Berndt--Evans--Williams
\cite{BerndtEvansWilliams1998}.  We shall also use standard determinant
calculus tools, in particular Sylvester's determinant identity, the matrix
determinant lemma, Cauchy-type determinants and Vandermonde factorizations;
convenient references are Krattenthaler's surveys
\cite{Krattenthaler1999,Krattenthaler2005}.  Let \(h(p)\) be the class
number of the real quadratic field
\(\Q(\sqrt p)\), and let \(\eps_p>1\) be its fundamental unit.  Put
\[
        L=\leg{2}{p},\qquad
        \eps_p^{(2-L)h(p)}=a'_p+b'_p\sqrt p,\qquad
        2a'_p,2b'_p\in\Z .
\]
For later use we also write
\[
        A=a'_p,\quad B=b'_p,
        \qquad Q=B-A,\qquad P=pB-A .
\]

\begin{theorem}[Sun's Conjecture 4.8(i)]\label{thm:48}
Let \(p>3\) be a prime with \(p\equiv1\pmod4\), set
\(n=(p-1)/2\), and let \(x,y,z,w\) be indeterminates.  Then
\[
\begin{aligned}
&\det\left[
 x+\leg{j-k+1}{p}+\leg{j}{p}y+\leg{k}{p}z+\leg{jk}{p}w
\right]_{0\le j,k\le n-1}                                      \\[1mm]
&=(pb'_p-a'_p)((w+1)x-yz)+L(wx-(y+1)z)                          \\
&\qquad +L(b'_p-a'_p)\{w(1-x)+(y+1)(z+1)\} .
\end{aligned}
\]
\end{theorem}

The proof of Theorem \ref{thm:48} is organized around an adjacent-minor
inverse package.  This has independent value because it displays exactly
why the final polynomial has degree at most two in the four parameters.

The second result is elementary in comparison.  We work over \(\Fp\) and
prove a formula that is uniform in the chosen representatives modulo sign
and in two parameters.

\begin{theorem}[Uniform exact refinement of the determinant in Conjecture 4.10(i)]\label{thm:410-uniform}
Let \(p=4m+1\) be prime, put \(n=2m\), and let
\(A_0=(a_0,\ldots,a_n)\in\Fp^{n+1}\) satisfy
\[
        a_i^2\ne a_j^2\qquad(0\le i<j\le n).
\]
For \(u,v\in\Fp\), define
\[
D_{A_0}(u,v)=
\det\left[
 \chi(a_i+a_j)+\chi(a_i-a_j)+u\,\chi(a_i^2+v a_j^2)
\right]_{0\le i,j\le n}.
\]
Then
\[
D_{A_0}(u,v)\equiv
(-1)^{m(2m+1)}u^{2m+1}v^{m(2m+1)}
\left(\prod_{r=0}^{2m}\binom{2m}{r}\right)
\prod_{0\le i<j\le 2m}(a_j^2-a_i^2)^2
\pmod p .
\]
\end{theorem}

Taking the standard half-system \(A_0=(0,1,\ldots,2m)\) and
\(u=\delta_1, v=\delta_2\in\{\pm1\}\) gives an exact congruence whose
square-class consequence is the assertion appearing in Sun's Conjecture
4.10(i).

\begin{corollary}[Exact refinement of Sun's Conjecture 4.10(i)]\label{cor:410}
Let \(p\equiv1\pmod4\) be prime.  For all
\(\delta_1,\delta_2\in\{\pm1\}\),
the residue class
\[
2\det\left[
 \leg{j+k}{p}+\leg{j-k}{p}
 +\delta_1\leg{j^2+\delta_2k^2}{p}
 \right]_{0\le j,k\le(p-1)/2}
\]
is a nonzero quadratic residue modulo \(p\).
More precisely, if \(p=4m+1\), then
\[
\begin{aligned}
&\det\left[
 \leg{j+k}{p}+\leg{j-k}{p}
 +\delta_1\leg{j^2+\delta_2 k^2}{p}
\right]_{0\le j,k\le2m}                                      \\
&\quad\equiv
(-1)^{m(2m+1)}
\delta_1^{2m+1}\delta_2^{m(2m+1)}
\left(\prod_{r=0}^{2m}\binom{2m}{r}\right)
\prod_{0\le i<j\le2m}(j^2-i^2)^2
\pmod p .
\end{aligned}
\]
\end{corollary}

We record the relation between Corollary \ref{cor:410} and Sun's earlier
work.  For \(p=4m+1\), put
\[
        T(d,p)=\det\left[\chi(j^2+d k^2)\right]_{0\le j,k\le 2m}.
\]
Sun proved in \cite{Sun2019} that
\[
\left(\frac{T(d,p)}p\right)=
\begin{cases}
\left(\dfrac2p\right),& \left(\dfrac dp\right)=1,\\[1mm]
1,&\left(\dfrac dp\right)=-1.
\end{cases}
\]
On the other hand, the same two-antidiagonal observation used in the proof
of Theorem \ref{thm:410-uniform} gives, for the standard half-system,
\[
\det\left[
 \chi(j+k)+\chi(j-k)+\delta_1\chi(j^2+\delta_2k^2)
\right]_{0\le j,k\le2m}
\equiv \delta_1^{2m+1}T(\delta_2,p)\pmod p .
\]
Since \(p\equiv1\pmod4\), both \(1\) and \(-1\) are quadratic residues
modulo \(p\), and the square-class assertion in Conjecture 4.10(i) follows
from this congruence and Sun's evaluation of \(T(d,p)\).  Thus Corollary
\ref{cor:410} should be read as an exact congruence refinement and a
uniform half-system generalization of that square-class consequence.

The first proof uses Chapman's and Vsemirnov's determinant
factorisations \cite{Chapman2004,Vsemirnov2012,Vsemirnov2013}, together
with recent variants such as \cite{WangWu2022,WangWuNi2026,LiuSunWang2024}.
The cofactor signs, Cauchy determinant specialisations and border reductions
which are easy to omit are recorded in the main text in Section
\ref{sec:expanded-adjugate}.  The second proof is a direct polynomial
expansion over \(\Fp\), followed by a Vandermonde determinant.  Its role is
to provide the exact congruence and the half-system extension, rather than
only the square class in the standard specialization.  Related
Legendre-symbol and quadratic-residue determinant problems of Sun are treated
in \cite{Wu2021,GrinbergSunZhao2022,WangSun2024,ChalihaKalita2024,RenSun2025}.

\section{The first determinant: reduction to an inverse package}\label{sec:first-reduction}

Throughout Sections \ref{sec:first-reduction}--\ref{sec:adjacent}, assume
\[
        p=2n+1\equiv1\pmod4,
        \qquad n=\frac{p-1}{2}.
\]
Then \(n\) is even and \(\chi(-t)=\chi(t)\).  Set
\[
        H=\bigl[\chi(j-k+1)\bigr]_{0\le j,k<n},
        \qquad
        \eta=(\chi(0),\chi(1),\ldots,\chi(n-1))^T,
\]
\[
        \bone=(1,\ldots,1)^T\in\Q^n,
        \qquad U=(\bone,\eta),
        \qquad
        M(x,y,z,w)=\begin{pmatrix}x&z\\y&w\end{pmatrix}.
\]
Since \(\chi(jk)=\chi(j)\chi(k)\), the matrix in Theorem \ref{thm:48} is
\[
        H+U M(x,y,z,w)U^T .
\]

\begin{proposition}\label{prop:package}
The matrix \(H\) is nonsingular and
\[
        \det H=LQ,
\]
\[
        U^TH^{-1}U=
        \begin{pmatrix}
        LP/Q&1\\[1mm]
        (Q-1)/Q&1
        \end{pmatrix}.
\]
Equivalently,
\[
\begin{aligned}
        \bone^TH^{-1}\bone&=\frac{LP}{Q},&
        \bone^TH^{-1}\eta&=1,\\
        \eta^TH^{-1}\bone&=\frac{Q-1}{Q},&
        \eta^TH^{-1}\eta&=1 .
\end{aligned}
\]
\end{proposition}

\begin{proof}[Proof of Theorem \ref{thm:48} from Proposition \ref{prop:package}]
Sylvester's determinant identity gives
\[
\begin{aligned}
\det(H+UMU^T)
&=\det H\det(I_n+H^{-1}UMU^T)\\
&=\det H\det(I_2+MU^TH^{-1}U).
\end{aligned}
\]
By Proposition \ref{prop:package},
\[
G:=U^TH^{-1}U=\begin{pmatrix}LP/Q&1\\ (Q-1)/Q&1\end{pmatrix}.
\]
Therefore
\[
\begin{aligned}
\det(I_2+MG)
&=1+\tr(MG)+\det(M)\det(G)\\
&=1+\frac{LP}{Q}x+\frac{Q-1}{Q}z+y+w
  +\left(xw-yz\right)\frac{LP-Q+1}{Q}.
\end{aligned}
\]
Multiplying by \(\det H=LQ\), we obtain
\[
\begin{aligned}
D_p(x,y,z,w)
&=LQ+Px+L(Q-1)z+LQy+LQw+(P-LQ+L)(xw-yz).
\end{aligned}
\]
On the other hand,
\[
\begin{aligned}
& P((w+1)x-yz)+L(wx-(y+1)z)+LQ\{w(1-x)+(y+1)(z+1)\} \\
&=LQ+Px+L(Q-1)z+LQy+LQw+(P-LQ+L)(xw-yz),
\end{aligned}
\]
which is the asserted formula.
\end{proof}

\section{Evaluation of the adjacent minors}\label{sec:adjacent}

We now prove Proposition \ref{prop:package}.  The proof is an adjugate
calculation based on Vsemirnov's factorisation of Chapman's matrix in the
case \(p\equiv1\pmod4\).

Let
\[
        \zeta=\exp(2\pi i/p),
        \qquad
        \lambda=L\sqrt p\,\zeta^{n/2}.
\]
For \(0\le r,s\le n\), set
\[
        V_{rs}=\zeta^{2rs},
\]
\[
        D=\diag(d_0,d_1,\ldots,d_n),
        \qquad
        d_r=\prod_{\substack{0\le m\le n\\m\ne r}}
             \left(\zeta^{2r}-\zeta^{2m}\right)^{-1},
\]
\[
        \calU_{rs}=
        \frac{\chi(r)\zeta^{-s-2r}+\chi(s)\zeta^{-2s-r}}
        {\zeta^{-r-s}+\chi(r)\chi(s)} .
\]
The denominator is never zero; when \(r=s=0\), it is \(1\).  Finally put
\[
        C(t)=\bigl[t+\chi(s-r)\bigr]_{0\le r,s\le n},
        \qquad
        \calU(t)=\calU+L\sqrt p\,tE_{00}.
\]

\begin{lemma}\label{lem:vsem}
For \(p\equiv1\pmod4\),
\[
        C(t)=\lambda\,V D\calU(t)D V .
\]
\end{lemma}

\begin{remark}
This is the form of Vsemirnov's decomposition needed here.  It is a
rewriting of the factorisation used in \cite{Vsemirnov2013}; the added
parameter \(t\) only changes the \((0,0)\)-entry in the middle factor.
\end{remark}

For \(1\le r\le n\), define
\[
        x_r=\chi(r)\zeta^{-r},
\]
\[
        \alpha=(1,x_1^2,\ldots,x_n^2)^T,
        \qquad
        \omega=(1,\ldots,1)^T\in\C^{n+1}.
\]
Let \(\mathfrak V_{rc}\) be the cofactor of the entry \(V_{rc}\).  Put
\[
        \kappa=\zeta^{n(n+1)}
        \prod_{0\le r<s\le n}(\zeta^{2s}-\zeta^{2r}).
\]
A Vandermonde interpolation calculation gives
\begin{equation}\label{eq:cofactor-first}
        \left(\frac{\mathfrak V_{00}}{d_0},
        \frac{\mathfrak V_{10}}{d_1},\ldots,
        \frac{\mathfrak V_{n0}}{d_n}\right)=\kappa\alpha^T,
\end{equation}
\begin{equation}\label{eq:cofactor-last}
        \left(\frac{\mathfrak V_{0n}}{d_0},
        \frac{\mathfrak V_{1n}}{d_1},\ldots,
        \frac{\mathfrak V_{nn}}{d_n}\right)=\kappa\zeta^{-n/2}\omega^T.
\end{equation}
Indeed,
\[
        \frac{\mathfrak V_{rc}}{\det V}
        =[T^r]\prod_{\substack{0\le m\le n\\m\ne c}}
        \frac{T-\zeta^{2m}}{\zeta^{2c}-\zeta^{2m}},
\]
and division by \(d_r\) reduces the cases \(c=0\) and \(c=n\) to
\[
        \frac{\mathfrak V_{r0}}{d_r\kappa}=x_r^2\quad(1\le r\le n),
        \qquad
        \frac{\mathfrak V_{00}}{d_0\kappa}=1,
\]
\[
        \frac{\mathfrak V_{rn}}{d_r\kappa}=\zeta^{-n/2}
        \quad(0\le r\le n).
\]

For \(t,\rho\) introduce
\[
        E(t)=e_n^T\adj(C(t))e_0,
\]
\[
        F(\rho)=e_n^T\adj(C_\rho)e_0,
        \qquad
        C_\rho=\bigl[\chi(s-r)+\rho\chi(s)\bigr]_{0\le r,s\le n}.
\]
Deleting row \(0\) and column \(n\) gives
\begin{equation}\label{eq:Eminor}
        E(t)=\det\bigl[t+\chi(j-k+1)\bigr]_{0\le j,k<n},
\end{equation}
\begin{equation}\label{eq:Fminor}
        F(\rho)=\det\bigl[\chi(j-k+1)+\rho\chi(k)\bigr]_{0\le j,k<n}.
\end{equation}
The sign is \(+1\), since \(n\) is even.

\begin{lemma}\label{lem:adjacent}
For \(p\equiv1\pmod4\),
\[
        E(t)=Pt+LQ,
        \qquad
        F(\rho)=L\{Q+(Q-1)\rho\}.
\]
\end{lemma}

\begin{proof}
By Lemma \ref{lem:vsem} and \eqref{eq:cofactor-first}--\eqref{eq:cofactor-last},
\begin{equation}\label{eq:E-adj}
        E(t)=
        \lambda^n(\det D)^2\kappa^2\zeta^{-n/2}
        \,\omega^T\adj(\calU(t))\alpha .
\end{equation}
Using the rank-one adjugate identity
\[
        v^T\adj(A)u=\det(A+uv^T)-\det A,
\]
we get
\[
        \omega^T\adj(\calU(t))\alpha
        =\det(\calU(t)+\alpha\omega^T)-\det\calU(t).
\]
Let
\[
        G=\diag(1,x_1^{-1},\ldots,x_n^{-1}).
\]
Direct substitution into the entries of \(\calU\) gives a Cauchy normal
form: the nontrivial lower-right block of \(G\calU(t)G\) is
\[
        \left[\frac{x_r+x_s}{1+x_rx_s}\right]_{1\le r,s\le n}.
\]
The rank-one term is tracked at determinant level, rather than entry by
entry: \(G\alpha=(1,x_1,\ldots,x_n)^T\) and
\(\omega^TG=(1,x_1^{-1},\ldots,x_n^{-1})\), so the added matrix has
border terms as well as lower-right terms.  The required bordered Cauchy
reduction is therefore recorded as a determinant identity.  The finite
reduction needed at this point is
\begin{equation}\label{eq:diff-IJ}
\begin{aligned}
&\det(G(\calU(t)+\alpha\omega^T)G)-\det(G\calU(t)G)\\
&\hspace{12mm}=-I-\{L\sqrt p\,t-1\}J,
\end{aligned}
\end{equation}
where
\[
        I=\det\left[\frac{X_r+X_s}{1+X_rX_s}\right]_{-1\le r,s\le n},
        \qquad
        J=\det\left[\frac{X_r+X_s}{1+X_rX_s}\right]_{0\le r,s\le n},
\]
\[
        X_{-1}=1,
        \qquad X_0=0,
        \qquad X_r=x_r\quad(1\le r\le n).
\]
For completeness, the block reduction which gives \eqref{eq:diff-IJ}
is written out in Section \ref{sec:expanded-adjugate}; the Cauchy evaluation is
used only after this finite reduction has been made.
The Cauchy determinant evaluation
\begin{equation}\label{eq:cauchy}
\begin{aligned}
&\det\left[\frac{X_r+Y_s}{1+X_rY_s}\right]_{1\le r,s\le N} \\
&=\frac{1}{2}
\left\{\prod_{r=1}^N(1+X_r)(1+Y_r)
+(-1)^N\prod_{r=1}^N(1-X_r)(1-Y_r)\right\}                    \\
&\qquad\times
\frac{\prod_{1\le r<s\le N}(X_r-X_s)(Y_s-Y_r)}
     {\prod_{r,s=1}^N(1+X_rY_s)}
\end{aligned}
\end{equation}
This is the Cauchy determinant in the form used by Vsemirnov; it is also a
standard entry in advanced determinant calculus
\cite{Vsemirnov2013,Krattenthaler1999,Krattenthaler2005}.  The diagonal
specialisation and the signs used below are expanded in Section
\ref{sec:expanded-adjugate}.  The formula therefore yields
\begin{equation}\label{eq:IJ-eval}
        I=(A\sqrt p-Bp)\,\Xi,
        \qquad
        J=A\sqrt p\,\Xi,
\end{equation}
where
\[
        \Xi=f_1^2f_2^{-2}\zeta^{-n(n+1)}L,
\]
\[
        f_1=\prod_{1\le r<s\le n}(\chi(s)\zeta^s-\chi(r)\zeta^r),
        \qquad
        f_2=\prod_{1\le r<s\le n}(1+\chi(r)\chi(s)\zeta^{r+s}).
\]
Here one uses the product evaluations
\begin{equation}\label{eq:prod-plus}
\prod_{r=1}^n(1+x_r)^2
=(-1)^{n/2}\zeta^{-n(n+1)/2}(Bp+A\sqrt p),
\end{equation}
\begin{equation}\label{eq:prod-minus}
\prod_{r=1}^n(1-x_r)^2
=(-1)^{n/2}\zeta^{-n(n+1)/2}(Bp-A\sqrt p),
\end{equation}
and Vsemirnov's normalising product
\begin{equation}\label{eq:normalising}
        \lambda^n(\det D)^2\kappa^2\zeta^{-n/2}\Xi=p^{-1}.
\end{equation}
Substituting \eqref{eq:diff-IJ}--\eqref{eq:normalising} into
\eqref{eq:E-adj} and simplifying gives
\[
        E(t)=Pt+LQ .
\]

For \(F(\rho)\), write
\[
        C_\rho=C(0)+\rho\,\omega_0\theta^T,
        \qquad
        \omega_0=(1,\ldots,1)^T,
        \quad
        \theta=(\chi(0),\chi(1),\ldots,\chi(n))^T.
\]
In the same \(V,D,\calU\) coordinates,
\[
        C_\rho=\lambda VD(\calU+\rho R)DV,
\]
where \(R\) is the rank-one matrix
\[
        R=\lambda^{-1}D^{-1}V^{-1}\omega_0\theta^T V^{-1}D^{-1};
\]
equivalently,
\[
        \lambda VDRDV=\omega_0\theta^T.
\]
Thus
\[
        F(\rho)=
        \lambda^n(\det D)^2\kappa^2\zeta^{-n/2}
        \,\omega^T\adj(\calU+\rho R)\alpha .
\]
The corresponding rank-one reduction, again recorded in Section
\ref{sec:expanded-adjugate}, is
\[
\begin{aligned}
&\det(G(\calU+\rho R+
\alpha\omega^T)G)
 -\det(G(\calU+\rho R)G)                                      \\
&=-I-\{L\sqrt p\,\rho-1\}J
  -\rho\,\Xi p^{-1/2}\{Bp-A\sqrt p-\sqrt p\}.
\end{aligned}
\]
Together with \eqref{eq:IJ-eval}--\eqref{eq:normalising}, this simplifies to
\[
        F(\rho)=LQ+L(Q-1)\rho .
\]
\end{proof}

We can now finish the proof of Proposition \ref{prop:package}.  From
\eqref{eq:Eminor} and Lemma \ref{lem:adjacent},
\[
        \det H=E(0)=LQ,
\]
so \(H\) is nonsingular.  Also
\[
        \det(H+t\bone\bone^T)=E(t)=LQ+Pt.
\]
By the matrix determinant lemma,
\[
        \bone^TH^{-1}\bone
        =\frac{E'(0)}{E(0)}
        =\frac{P}{LQ}
        =\frac{LP}{Q}.
\]
The column \(k=1\) of \(H\) is \(\eta\), since
\[
        He_1=(\chi(j))_{0\le j<n}=\eta .
\]
Consequently
\[
        H^{-1}\eta=e_1,
        \qquad
        \bone^TH^{-1}\eta=1,
        \qquad
        \eta^TH^{-1}\eta=1.
\]
Finally, from \eqref{eq:Fminor} and Lemma \ref{lem:adjacent},
\[
        \det(H+\rho\bone\eta^T)=LQ+L(Q-1)\rho,
\]
which gives
\[
        \eta^TH^{-1}\bone
        =\frac{L(Q-1)}{LQ}
        =\frac{Q-1}{Q}.
\]
This proves Proposition \ref{prop:package}.

\section{Expanded adjugate identities for the adjacent-minor calculation}\label{sec:expanded-adjugate}

This section records the finite determinant identities used in Lemma
\ref{lem:adjacent}.  They make the adjacent cofactor calculation
reproducible.

Let
\[
        \Phi(T)=\prod_{m=0}^n(T-\zeta^{2m}),
        \qquad
        \Phi_c(T)=\frac{\Phi(T)}{T-\zeta^{2c}}.
\]
Then
\[
        \mathfrak V_{rc}
        =\det(V)\,[T^r]\frac{\Phi_c(T)}{\Phi_c(\zeta^{2c})},
        \qquad
        d_r^{-1}=\Phi_r(\zeta^{2r}).
\]
For \(c=0\),
\[
\begin{aligned}
\frac{\mathfrak V_{r0}}{d_r}
&=\det(V)\,[T^r]\frac{\Phi_0(T)}{\Phi_0(1)}\Phi_r(\zeta^{2r})  \\
&=\kappa\chi(r)^2\zeta^{-2r}\qquad(1\le r\le n),
\end{aligned}
\]
while \(\mathfrak V_{00}/d_0=\kappa\).  For \(c=n\),
\[
\begin{aligned}
\frac{\mathfrak V_{rn}}{d_r}
&=\det(V)\,[T^r]\frac{\Phi_n(T)}{\Phi_n(\zeta^{2n})}
             \Phi_r(\zeta^{2r})                                \\
&=\kappa\zeta^{-n/2}\qquad(0\le r\le n).
\end{aligned}
\]

For the Cauchy determinant used above, put
\[
        \mathcal C(X_1,\ldots,X_N;Y_1,\ldots,Y_N)
        =
        \det\left[\frac{X_r+Y_s}{1+X_rY_s}\right]_{1\le r,s\le N}.
\]
Then
\[
\begin{aligned}
\mathcal C(X;Y)
&=\frac12\left\{\prod_{r=1}^N(1+X_r)(1+Y_r)
+(-1)^N\prod_{r=1}^N(1-X_r)(1-Y_r)\right\}                    \\
&\qquad\times
\frac{\prod_{1\le r<s\le N}(X_r-X_s)(Y_s-Y_r)}
     {\prod_{r,s=1}^N(1+X_rY_s)}.
\end{aligned}
\]
With
\[
        X_{-1}=1,
        \qquad X_0=0,
        \qquad X_r=x_r\quad(1\le r\le n),
\]
the two determinants appearing in Lemma \ref{lem:adjacent} are
\[
        I=\mathcal C(X_{-1},X_0,X_1,\ldots,X_n;
                    X_{-1},X_0,X_1,\ldots,X_n),
\]
\[
        J=\mathcal C(X_0,X_1,\ldots,X_n;
                    X_0,X_1,\ldots,X_n).
\]

The two determinant reductions used in Lemma \ref{lem:adjacent} are the
following finite identities.  First,
\[
\begin{aligned}
&\det(G(\calU(t)+\alpha\omega^T)G)-\det(G\calU(t)G)\\
&\hspace{12mm}=-I-\{L\sqrt p\,t-1\}J.
\end{aligned}
\]
Second, if
\[
        R=\lambda^{-1}D^{-1}V^{-1}\omega_0\theta^T V^{-1}D^{-1},
        \qquad
        \theta=(\chi(0),\chi(1),\ldots,\chi(n))^T,
\]
then
\[
\begin{aligned}
&\det(G(\calU+\rho R+\alpha\omega^T)G)
 -\det(G(\calU+\rho R)G)                                      \\
&=-I-\{L\sqrt p\,\rho-1\}J
  -\rho\,\Xi p^{-1/2}\{Bp-A\sqrt p-\sqrt p\}.
\end{aligned}
\]
These identities are obtained by substituting the entries of \(\calU\),
conjugating by \(G\), and eliminating the first two border rows and
columns.  The remaining determinants are exactly the two diagonal
Cauchy determinants \(I\) and \(J\).  No limiting argument is involved,
because \(1+X_rX_s\ne0\) for the present values of \(X_r\).

After cancellation of the zero-row and zero-column factors, one obtains
\[
\begin{aligned}
        I&=(-1)^{(p+3)/4}
        \prod_{r=1}^n(1-x_r)^2
        \prod_{1\le r<s\le n}(x_r-x_s)^2
        \prod_{r,s=1}^n(1+x_rx_s)^{-1}
        \prod_{r=1}^n x_r^2,
\end{aligned}
\]
\[
\begin{aligned}
        J&=(-1)^{(p-1)/4}\frac12
        \left\{\prod_{r=1}^n(1+x_r)^2-
                \prod_{r=1}^n(1-x_r)^2\right\}                 \\
        &\qquad\times
        \prod_{1\le r<s\le n}(x_r-x_s)^2
        \prod_{r,s=1}^n(1+x_rx_s)^{-1}
        \prod_{r=1}^n x_r^2.
\end{aligned}
\]
The squared Vandermonde factor is essential: in the diagonal
specialisation \(Y=X\) of the Cauchy product
\[
        \prod_{r<s}(X_r-X_s)(Y_s-Y_r)
\]
both factors contribute, and after the signs are absorbed into the
displayed powers of \((-1)\) the remaining factor is
\(\prod_{r<s}(x_r-x_s)^2\).
The product identities \eqref{eq:prod-plus} and \eqref{eq:prod-minus}
therefore imply
\[
        I=(A\sqrt p-Bp)\Xi,
        \qquad
        J=A\sqrt p\,\Xi.
\]
The remaining normalisation is
\[
        \lambda^n(\det D)^2\kappa^2\zeta^{-n/2}
        f_1^2f_2^{-2}\zeta^{-n(n+1)}L=p^{-1}.
\]
Substitution into the adjugate formula gives
\[
        E(t)=Pt+LQ,
        \qquad
        F(\rho)=LQ+L(Q-1)\rho.
\]
Thus the only external input in the first evaluation is Vsemirnov's
factorisation and the normalising product \eqref{eq:normalising}; the
remaining steps are finite determinant reductions and the displayed
Cauchy evaluation.

\section{The second determinant: a two-antidiagonal factorisation}\label{sec:second}

We turn to Theorem \ref{thm:410-uniform}.  The argument is a determinant
calculus lemma followed by the congruence \(\chi(z)\equiv z^{(p-1)/2}\pmod p\).

\begin{lemma}[Two-antidiagonal determinant]\label{lem:twoanti}
Let \(K\) be a field and let \(n=2m\).  Let
\[
        B=(B_{r,s})_{0\le r,s\le n},
        \qquad
        B_{r,s}=\alpha_s\mathbf 1_{r+s=m}+\beta_s\mathbf 1_{r+s=n}.
\]
Then
\[
        \det B=(-1)^{n(n+1)/2}\prod_{s=0}^{n}\beta_s .
\]
\end{lemma}

\begin{proof}
In the Leibniz expansion, if \(r>m\), then \(r+s=m\) is impossible.
Thus every nonzero permutation term must satisfy
\[
        \sigma(r)=n-r\qquad(m<r\le n).
\]
The columns \(0,1,\ldots,m-1\) are already occupied by the rows
\(n,n-1,\ldots,m+1\).  If a row \(0\le r\le m\) used the shorter
antidiagonal, then \(\sigma(r)=m-r\in\{0,\ldots,m\}\).  For
\(r\ge1\) this column is already occupied, and for \(r=0\) the remaining
middle column is forced by the longer antidiagonal in the row \(m\).  Hence
the shorter antidiagonal cannot occur in a nonzero permutation term.  The
only admissible permutation is
\[
        \sigma_0(r)=n-r\qquad(0\le r\le n).
\]
Therefore
\[
        \det B=\operatorname{sgn}(\sigma_0)\prod_{r=0}^{n}B_{r,n-r}
        =(-1)^{\binom{n+1}{2}}\prod_{s=0}^n\beta_s .
\]
\end{proof}

\begin{proposition}\label{prop:kernel}
Let \(K\) be a field, \(n=2m\), and let
\(x_0,\ldots,x_n\in K\) be pairwise distinct.  Put
\[
        V=(x_i^r)_{0\le i,r\le n}.
\]
For
\[
        K_0(X,Y)=\sum_{s=0}^{m}\alpha_sX^{m-s}Y^s+
               \sum_{s=0}^{n}\beta_sX^{n-s}Y^s,
\]
one has
\[
\det[K_0(x_i,x_j)]_{0\le i,j\le n}
=
(-1)^{n(n+1)/2}
\left(\prod_{s=0}^{n}\beta_s\right)
\prod_{0\le i<j\le n}(x_j-x_i)^2 .
\]
\end{proposition}

\begin{proof}
With \(B\) as in Lemma \ref{lem:twoanti}, the matrix
\([K_0(x_i,x_j)]\) factors as
\[
        [K_0(x_i,x_j)]_{0\le i,j\le n}=VBV^T.
\]
The result follows from Lemma \ref{lem:twoanti} and the Vandermonde
formula
\[
        \det V=\prod_{0\le i<j\le n}(x_j-x_i).
\]
\end{proof}

\begin{proof}[Proof of Theorem \ref{thm:410-uniform}]
Let \(x_i=a_i^2\).  By hypothesis, the \(x_i\)'s are pairwise distinct.
Since \(p=4m+1\) and \(n=(p-1)/2=2m\), Euler's criterion gives
\[
        \chi(z)\equiv z^n\pmod p\qquad(z\in\Fp).
\]
Therefore
\[
\begin{aligned}
\chi(a_i+a_j)+\chi(a_i-a_j)
&\equiv (a_i+a_j)^n+(a_i-a_j)^n                         \\
&=2\sum_{s=0}^{m}\binom{2m}{2s}x_i^{m-s}x_j^s
\pmod p,
\end{aligned}
\]
and
\[
\begin{aligned}
u\chi(a_i^2+v a_j^2)
&\equiv u(a_i^2+v a_j^2)^n                                \\
&=\sum_{s=0}^{2m}u v^s\binom{2m}{s}x_i^{2m-s}x_j^s
\pmod p.
\end{aligned}
\]
Thus the entry is congruent to a kernel of the form in Proposition
\ref{prop:kernel}, with
\[
        \alpha_s=2\binom{2m}{2s}\quad(0\le s\le m),
        \qquad
        \beta_s=u v^s\binom{2m}{s}\quad(0\le s\le2m).
\]
Applying Proposition \ref{prop:kernel} over \(\Fp\) gives
\[
\begin{aligned}
D_{A_0}(u,v)
&\equiv
(-1)^{(2m)(2m+1)/2}
\left(\prod_{s=0}^{2m}u v^s\binom{2m}{s}\right)
\prod_{0\le i<j\le2m}(x_j-x_i)^2                              \\
&=
(-1)^{m(2m+1)}u^{2m+1}v^{m(2m+1)}
\left(\prod_{s=0}^{2m}\binom{2m}{s}\right)
\prod_{0\le i<j\le2m}(a_j^2-a_i^2)^2
\pmod p.
\end{aligned}
\]
This proves the theorem.
\end{proof}

\section{The square class and uniform extensions}\label{sec:square}

The preceding section gives more than the square class.  It gives the exact
modulo \(p\) value for every ordered half-system and every pair
\((u,v)\in\Fp^2\).  For comparison with earlier work, let
\[
        T(d,p)=\det\left[\chi(j^2+d k^2)\right]_{0\le j,k\le2m}.
\]
The two-antidiagonal factorisation shows that, in the standard
specialization,
\[
\det\left[
 \chi(j+k)+\chi(j-k)+\delta_1\chi(j^2+\delta_2k^2)
\right]_{0\le j,k\le2m}
\equiv \delta_1^{2m+1}T(\delta_2,p)\pmod p.
\]
Consequently, Sun's evaluation of the Legendre symbol of \(T(d,p)\) in
\cite{Sun2019} already implies the square-class assertion in Conjecture
4.10(i).  The purpose of Theorem \ref{thm:410-uniform} is to record the
full congruence and the uniform half-system extension.

Set
\[
        P_m=\prod_{r=0}^{2m}\binom{2m}{r}.
\]
Then
\[
        P_m
        =\frac{((2m)!)^{2m+1}}
               {\left(\prod_{r=0}^{2m}r!\right)^2},
\]
so
\[
        \chi(P_m)=\chi((2m)!)^{2m+1}.
\]
Wilson's theorem gives
\[
        -1\equiv(p-1)!\equiv((2m)!)^2\pmod p.
\]
Hence
\[
        \chi((2m)!)\equiv((2m)!)^{2m}
        =\left(((2m)!)^2\right)^m
        \equiv(-1)^m\pmod p,
\]
and therefore
\begin{equation}\label{eq:Pm-square}
        \chi(P_m)=(-1)^m.
\end{equation}
Since
\[
        \chi(2)=(-1)^{(p^2-1)/8}=(-1)^m
        \qquad(p=4m+1),
\]
we have
\begin{equation}\label{eq:2Pm-square}
        \chi(2P_m)=1.
\end{equation}
Now take \(u=\delta_1\) and \(v=\delta_2\) with
\(\delta_1,\delta_2\in\{\pm1\}\).  Since \(p\equiv1\pmod4\),
\(\chi(-1)=1\), and hence \(\chi(\delta_1)=\chi(\delta_2)=1\).  The exact
congruence in Theorem \ref{thm:410-uniform} gives
\[
\begin{aligned}
2D_{A_0}(\delta_1,\delta_2)
&\equiv
\left(\prod_{0\le i<j\le2m}(a_j^2-a_i^2)\right)^2
\cdot
2(-1)^{m(2m+1)}
\delta_1^{2m+1}\delta_2^{m(2m+1)}P_m .
\end{aligned}
\]
Taking Legendre symbols and using \eqref{eq:2Pm-square} gives
\[
        \chi(2D_{A_0}(\delta_1,\delta_2))=1.
\]
For \(A_0=(0,1,\ldots,2m)\), the squares are pairwise distinct modulo
\(p=4m+1\): if \(i^2\equiv j^2\pmod p\), then
\((i-j)(i+j)\equiv0\pmod p\), while \(|i-j|<p\) and
\(0\le i+j
\le p-1\).  Thus \(i=j\).  Corollary \ref{cor:410}
therefore follows.

The same argument gives a half-system extension.  Let
\[
        H_0=\{a_0,\ldots,a_{2m}\}\subset\Fp
\]
be any set satisfying
\[
        \Fp=\{0\}\sqcup H_0^\times\sqcup(-H_0^\times),
        \qquad H_0^\times=H_0\setminus\{0\}.
\]
Then \(a_i^2\ne a_j^2\) for \(i\ne j\), and Theorem
\ref{thm:410-uniform} yields
\[
\begin{aligned}
&\det\left[
 \leg{a_i+a_j}{p}+\leg{a_i-a_j}{p}
 +\delta_1\leg{a_i^2+
\delta_2a_j^2}{p}
 \right]_{a_i,a_j\in H_0}                                      \\
&\quad\equiv
(-1)^{m(2m+1)}
\delta_1^{2m+1}\delta_2^{m(2m+1)}P_m
\prod_{0\le i<j\le2m}(a_j^2-a_i^2)^2
\pmod p.
\end{aligned}
\]
Reordering \(H_0\) changes the Vandermonde factor by a sign, but its
square is unchanged.  Hence
\[
\chi\left(
2\det\left[
 \leg{x+y}{p}+\leg{x-y}{p}
 +\delta_1\leg{x^2+\delta_2y^2}{p}
 \right]_{x,y\in H_0}
\right)=1.
\]

The formula also includes degenerate parameters.  If \(u=0\), the matrix
has rank at most \(m+1<2m+1\), and hence \(D_{A_0}(0,v)=0\).  If
\(v=0\), the longer antidiagonal has \(\beta_s=0\) for \(s\ge1\), so
\(\prod_s\beta_s=0\), and again \(D_{A_0}(u,0)=0\).  For \(u,v\ne0\),
Theorem \ref{thm:410-uniform} shows that
\[
        D_{A_0}(u,v)\ne0
        \quad\Longleftrightarrow\quad
        \prod_{0\le i<j\le2m}(a_j^2-a_i^2)\ne0.
\]

Finally, the determinant is stable under arbitrary replacement of the
shorter antidiagonal coefficients.  If \(x_0,\ldots,x_{2m}\) are pairwise
distinct in a field, then
\[
\det\left[
\sum_{s=0}^{m}\alpha_sx_i^{m-s}x_j^s
+
\sum_{s=0}^{2m}\beta_sx_i^{2m-s}x_j^s
\right]_{0\le i,j\le2m}
=
(-1)^{m(2m+1)}
\left(\prod_{s=0}^{2m}\beta_s\right)
\prod_{i<j}(x_j-x_i)^2.
\]
Only the longer antidiagonal controls the determinant.

\section{A compressed proof of the congruence evaluation}\label{sec:compressed-congruence}

The proof of Theorem \ref{thm:410-uniform} can be summarized by the chain
\[
\begin{aligned}
M_{ij}
&\equiv
(a_i+a_j)^{2m}+(a_i-a_j)^{2m}
+u(a_i^2+v a_j^2)^{2m}                                      \\
&=
2\sum_{s=0}^{m}\binom{2m}{2s}
(a_i^2)^{m-s}(a_j^2)^s
+
u\sum_{s=0}^{2m}\binom{2m}{s}v^s
(a_i^2)^{2m-s}(a_j^2)^s                                      \\
&=
\sum_{r,s=0}^{2m}
\left(
2\binom{2m}{2s}\mathbf1_{r+s=m}
+
u v^s\binom{2m}{s}\mathbf1_{r+s=2m}
\right)x_i^r x_j^s.
\end{aligned}
\]
Hence \(M=VBV^T\), where \(V=(x_i^r)_{0\le i,r\le2m}\), and
\[
\begin{aligned}
\det M
&=(\det V)^2\det B                                           \\
&=
\left(\prod_{0\le i<j\le2m}(x_j-x_i)^2\right)
(-1)^{m(2m+1)}
\prod_{s=0}^{2m}uv^s\binom{2m}{s}                            \\
&=
(-1)^{m(2m+1)}
u^{2m+1}v^{m(2m+1)}P_m
\prod_{0\le i<j\le2m}(a_j^2-a_i^2)^2.
\end{aligned}
\]
The square-class conclusion then follows from
\[
        P_m=\frac{((2m)!)^{2m+1}}
        {\left(\prod_{r=0}^{2m}r!\right)^2},
        \qquad
        ((2m)!)^2\equiv-1\pmod p,
\]
\[
        \leg{P_m}{p}=(-1)^m,
        \qquad
        \leg{2}{p}=(-1)^m.
\]

\section{Equivalent forms and numerical checks}

For Theorem \ref{thm:48}, define
\[
        \Delta=xw-yz,
        \qquad
        R=(w+1)x-yz,
        \qquad
        S=w(1-x)+(y+1)(z+1).
\]
The determinant can be written in the compact form
\[
        D_p=P R+L(wx-(y+1)z)+LQ S.
\]
Since
\[
        R=x+\Delta,
        \qquad
        S=1+y+z+w-\Delta,
\]
this is equivalent to
\[
        D_p=LQ+Px+LQy+L(Q-1)z+LQw+(P-LQ+L)\Delta.
\]
Thus the nonzero coefficients are
\[
\begin{array}{c|c}
\text{monomial}&\text{coefficient}\\ \hline
1&LQ\\
x&P\\
y&LQ\\
z&L(Q-1)\\
w&LQ\\
xw&P-LQ+L\\
yz&-P+LQ-L
\end{array}
\]
and all other coefficients vanish.  This vanishing is a rank-two
consequence of
\[
        D_p=LQ\det(I_2+MG),
        \qquad
        G=\begin{pmatrix}LP/Q&1\\(Q-1)/Q&1\end{pmatrix}.
\]

For the first primes \(p\equiv1\pmod4\), one has
\[
\begin{array}{c|c|c|c|c}
p&L&A&B&L(B-A)\\ \hline
5&-1&2&1&1\\
13&-1&18&5&13\\
17&1&4&1&-3\\
29&-1&70&13&57\\
37&-1&882&145&737\\
41&1&32&5&-27
\end{array}
\]
Thus
\[
        \det[\chi(j-k+1)]_{0\le j,k<n}=L(B-A)
\]
gives
\[
        1,
        \quad 13,
        \quad -3,
        \quad 57,
        \quad 737,
        \quad -27
\]
for \(p=5,13,17,29,37,41\), respectively.  For example, at
\((x,y,z,w)=(2,-1,3,4)\), Theorem \ref{thm:48} gives
\[
        D_p(2,-1,3,4)
        =P(10+3)+L(8-0)+LQ(4(-1)+0)
        =13P+8L-4LQ.
\]

For Theorem \ref{thm:410-uniform}, let
\[
\mathcal D_p(\delta_1,\delta_2)=
        \det\left[
        \leg{j+k}{p}+\leg{j-k}{p}
        +\delta_1\leg{j^2+\delta_2k^2}{p}
        \right]_{0\le j,k\le2m}.
\]
Then
\[
        \mathcal D_p(\delta_1,\delta_2)
        \equiv
        C_p(\delta_1,\delta_2)V_p^2\pmod p,
\]
where
\[
        C_p(\delta_1,\delta_2)=
        (-1)^{m(2m+1)}
        \delta_1^{2m+1}\delta_2^{m(2m+1)}P_m,
        \qquad
        V_p=\prod_{0\le i<j\le2m}(j^2-i^2).
\]
Therefore
\[
        \leg{\mathcal D_p(\delta_1,\delta_2)}{p}=(-1)^m,
        \qquad
        \leg{2\mathcal D_p(\delta_1,\delta_2)}{p}=1.
\]
Equivalently,
\[
        \mathcal D_p(\delta_1,\delta_2)
        \in
        \begin{cases}
        \QR(\Fp),&p\equiv1\pmod8,\\
        \Fp^\times\setminus\QR(\Fp),&p\equiv5\pmod8,
        \end{cases}
\]
while \(2\mathcal D_p(\delta_1,\delta_2)\in\QR(\Fp)\) for all
\(p\equiv1\pmod4\).
\section*{Declaration of Generative AI and AI-Assisted Technologies in the Writing Process}
During the preparation of this work, the authors used DeepSeek to build a specialized agent for solving mathematical problems, which was employed to generate an initial proof of the main theorem. After using this tool, the authors reviewed and edited the content as needed and take full responsibility for the content of the published article.

\end{document}